\documentclass[12pt,a4paper]{article}

\usepackage{amsfonts}
\usepackage{amsmath}
\usepackage{amssymb}
\usepackage{amstext}
\usepackage{amsthm}
\usepackage{mathrsfs}

\newcommand{\R}{\mathbb R}

\newcommand{\eps}{\varepsilon}
\newtheorem{teo}{Theorem}

\newtheorem{lemma}[teo]{Lemma}
\newtheorem{prop}[teo]{Proposition}

\newtheorem{defin}{Definition}
\theoremstyle{remark}
\newtheorem{rem}[teo]{Remark}

\DeclareMathOperator{\Id}{Id}
\DeclareMathOperator{\cat}{cat}

\begin{document}

\title{Positive solutions of singularly perturbed nonlinear elliptic problem
on Riemannian manifolds with boundary}
\author{M.Ghimenti
\thanks{Dipartimento di Matematica Applicata, Universit\`a di Pisa, via Buonarroti
1c, 56127, Pisa, Italy, e-mail \texttt{ghimenti@mail.dm.unipi.it,
a.micheletti@dma.it} }, A.M.Micheletti\addtocounter{footnote}{-1}
\footnotemark }
\date{}
\maketitle

\begin{abstract}
Let $(M,g)$ be a smooth connected compact Riemannian manifold of finite
dimension $n\geq 2$\ with a smooth boundary $\partial M$. We consider the
problem
\begin{equation*}
-\varepsilon ^{2}\Delta _{g}u+u=|u|^{p-2}u,\ u>0\text{ on }M,\
\frac{\partial u}{\partial \nu }=0\text{ on }\partial M
\end{equation*}
where $\nu $ is an exterior normal to $\partial M$.

The number of solutions of this problem depends on the topological
properties of the manifold. In particular we consider the Lusternik
Schnirelmann category of the boundary.

AMS subject: 58G03, 58E30
\end{abstract}

\section{Introduction}

Let $(M,g)$ be a smooth connected compact Riemannian manifold of finite
dimension $n\geq 2$\ with a smooth boundary $\partial M$, that is $\partial
M $ is  the union of a finite number of connected, smooth, boundaryless,
submanifold of $M$ of dimension $n-1$. Here $g$ denotes the Riemannian
metric tensor. By Nash theorem we can consider $(M,g)$ embedded as a regular
submanifold embedded in $\mathbb{R}^{N}$. We are interested in finding 
solutions $u\in H_{g}^{1}(M)$ of the following singularly perturbed 
nonlinear elliptic problem

\begin{equation}
\left\{
\begin{array}{cl}
-\varepsilon ^{2}\Delta _{g}u+u=|u|^{p-2}u,\text{ \ }u>0 & \text{ on }M \\
\frac{\partial u}{\partial \nu }=0 & \text{on }\partial M
\end{array}
\right.  \tag{$P$}  \label{P}
\end{equation}
for $2<p<2^{\ast }=\frac{2N}{N-2}$, where $\nu $ is the external normal to
$\partial M$.

Here $\displaystyle H_{g}^{1}(M)=\left\{ u:M\rightarrow \mathbb{R}:
\int_{M}|\nabla _{g}u|^{2}+u^{2}d\mu _{g}<\infty \right\} $ where $\mu _{g}$
denotes the volume form on $M$ associated to $g$.

Above type of equations have been extensively studied when $M$ is 
a flat bounded domain
$\Omega\subset \mathbb{R}^{N}$. We recall some classical result about the
Neumann problem in $\Omega$.
In \cite{LNT,NT1,NT2}, Lin, Ni and Takagi established the existence of
least-energy solution to (\ref{P}) and showed that for $\eps$
small enough the least energy
solution has a boundary spike.
Later, in \cite{DFW,W1} it was proved that for any stable
critical point of the mean curvature of the boundary it is
possible to construct single
boundary spike layer solutions, while in \cite{G,Li,WW}
the authors construct multiple boundary
spike solutions. Finally, in \cite{DY,GWW} the authors
proved that for any integer $K$ there exists a boundary
$K$-peaks solutions.

For which concerns the probem (\ref{P}) on a manifold $M$, with boundary
and without boundary,
Byeon and Park \cite{BP05} showed that the mountain pass solution
$u_\varepsilon$ has a spike layer.

A lot of works are devoted to show the influence of the topology of
$\Omega$ on the number of solutions of the Dirichlet problem
\begin{equation}
\left\{
\begin{array}{cl}
-\varepsilon ^{2}\Delta _{g}u+u=|u|^{p-2}u, u>0&\text{ on }\Omega\subset\R^N; \\
u=0 &\text{ on }\partial\Omega,
\end{array}
\right.
\end{equation}
when $\Omega$ is  a flat subset of $\R^N$.
We limit to cite \cite{BaCo88,BaL90,BaLi97,BC91,BCP91,BP05,BW02}.

Recently there have been some results on the effect of the topology
of the manifold $M$ on the number of solutions of the equation
$-\varepsilon ^{2}\Delta _{g}u+u=|u|^{p-2}u$ on a manifold
$M$ without boundary.
In \cite{BBM07} the authors proved that, if $M$ has a rich topology,
the equation has
multiple solutions. More precisely they show that this equation has at least
$\cat(M)+1$ positive nontrivial solutions for $\eps$ small enough.
Here $\cat(M)$ is the Lusternik-Schnirelmann category of $M$.
In \cite{Vta} there is the same result for a more general nonlinearity. Furthermore in \cite{Hta}
it was shown that the number of solution is influenced by the topology of a suitable
subset of $M$ depending on the geometry of $M$.

Our result concerns problem (\ref{P}) on \ a manifold $M$ with $\partial
M\neq \emptyset $. In this case we show that 
the topology of the boundary $\partial M$
influences the number of solutions, as follows.

\begin{teo}
For $\varepsilon $ small enough the problem (\ref{P}) has at least
$\cat(\partial M)+1$ non constant distinct solutions.
\end{teo}

The paper is organized as follows. In Section 2 we introduce some notions
and notations. In Section 3 we sketch the proof of the main result. The
details of the proof are in sections 4-7.

\section{Preliminaries}

We consider the $C^{2}$ functional defined on $H_{g}^{1}(M)$
\begin{equation}
J_{\varepsilon }(u)=\frac{1}{\varepsilon ^{N}}\int_{M}\left( \frac{1}{2}
\varepsilon ^{2}|\nabla _{g}u|^{2}+\frac{1}{2}|u|^{2}-\frac{1}{p}
|u^{+}|^{p}\right) d\mu _{g}.
\end{equation}
where $u^{+}(x)=\max \left\{ u(x),0\right\} $. It is well known that the
critical points of $J_{\varepsilon }(u)$ constrained on the associated
$C^{2} $ Nehari manifold
\begin{equation}
\mathcal{N}_{\varepsilon }=\left\{ u\in H_{g}^{1}\smallsetminus \{0\}\ :\
J_{\varepsilon }^{\prime }(u)u=0\right\}
\end{equation}
are non trivial solution of problem (\ref{P}).

Let $\mathbb{R}_{+}^{n}= \left\{ x=(\bar{x},x_{n}):\bar{x}\in
\mathbb{R}^{n-1},x_{n}\geq 0\right\}$. It is known that there exists a least energy
solution $V\in H^{1}(\mathbb{R}_{+}^{n})$ of the equation
\begin{equation}
\left\{
\begin{array}{cl}
-\Delta V+V=|V|^{p-2}V,\text{ }V>0 & \text{ on }\mathbb{R}_{+}^{n} \\
\frac{\partial V}{\partial x_{n}}|_{(\bar{x},0)}=0. &
\end{array}
\right.
\end{equation}
Moreover $V$ is radially symmetric and $\left\vert D^{\alpha}V(x)\right\vert
\leq c\exp (-\mu \left\vert x\right\vert )$ with $\left\vert \alpha
\right\vert \leq 2$, and $c,\mu $ positive constants.

If $V$ is a solution, also $V(x+y)$ with $y=(\bar{y},0)$ is a solution,
$V_{\varepsilon }(x)=V\left( \frac{x}{\varepsilon }\right) $ is a solution of
\begin{equation}
\left\{
\begin{array}{cl}
-\varepsilon ^{2}\Delta V_{\varepsilon }+V_{\varepsilon }=
|V_{\varepsilon}|^{p-2}V_{\varepsilon} & \text{ on }\mathbb{R}_{+}^{n} \\
\frac{\partial V_{\varepsilon }}{\partial x_{n}}|_{(\bar{x},0)}=0. &
\end{array}
\right.
\end{equation}

We put
\begin{equation}
m_{e}^{+}=\inf \left\{ E^{+}(v):v\in \mathcal{N}(E^{+})\right\} \text{ and }
m_{e}=\inf \left\{ E(v):v\in \mathcal{N}(E)\right\} ,
\end{equation}
where
\begin{eqnarray*}
E^{+}(v) &=&\int_{\mathbb{R}_{+}^{n}}\frac{1}{2}|\nabla v|^{2}+\frac{1}{2}
|v|^{2}-\frac{1}{p}|v^{+}|^{p}dx; \\
E(v) &=&\int_{\mathbb{R}^{n}}\frac{1}{2}|\nabla v|^{2}+\frac{1}{2}|v|^{2}-
\frac{1}{p}|v^{+}|^{p}dx.
\end{eqnarray*}
and
\begin{eqnarray*}
\mathcal{N}(E^{+}) &=&\left\{ v\in H^{1}(\mathbb{R}_{+}^{n})\smallsetminus
\{0\}\ :\ E^{+}(v)v=0\right\} ; \\
\mathcal{N}(E) &=&\left\{ v\in H^{1}(\mathbb{R}^{n})\smallsetminus \{0\}\ :\
E(v)v=0\right\} .
\end{eqnarray*}
It holds
\begin{equation*}
m_{e}=2m_{e}^{+},
\end{equation*}
and
\begin{equation*}
m_{e}^{+}=E^{+}(V)=\left( \frac{1}{2}-\frac{1}{p}\right) \left(
S_{e}^{+}\right) ^{\frac{p}{p-2}}\text{ where }S_{e}^{+}=\inf \left\{
\frac{\left\vert \left\vert v\right\vert \right\vert_{H^{1}
(\mathbb{R}_{+}^{n})}^{2}}{\left\vert \left\vert v\right\vert \right\vert _{L^{p}
(\mathbb{R}_{+}^{n})}^{2}},v\neq 0\right\} .
\end{equation*}

\begin{rem}
\label{remR} On the tangent bundle of any compact Riemannian manifold
$\mathcal{M}$ it is defined the exponential map $\exp :T\mathcal{M}
\rightarrow \mathcal{M}$ which is of class $C^{\infty }$. Moreover there
exists a constant $R>0$ and a finite number of $x_{i}\in \mathcal{M}$ such
that $\mathcal{M}=\cup _{i=1}^{l}B_{g}(x_{i},R)$ and $\exp
_{x_{i}}:B(0,R)\rightarrow B_{g}(x_{i},R)$ is a diffeormophism for all $i$.
\end{rem}

By choosing an orthogonal coordinate system $(y_{1},\dots ,y_{n})$ of
$\mathbb{R}^{n}$ and identifying $T_{x_{0}}\mathcal{M}$ with $\mathbb{R}^{n}$
for $x_{0}\in \mathcal{M}$ we can define by the exponential map the so
called normal coordinates. For $x_{0}\in \mathcal{M},$ $g_{x_{0}}$ denotes
the metric read through the normal coordinates. In particular, we have
$g_{x_{0}}(0)=\Id$. We set $\left\vert g_{x_{0}}(y)\right\vert =\det \left(
g_{x_{0}}(y)\right) _{ij}$ and $g_{x_{0}}^{ij}(y)=\left( \left(
g_{x_{0}}(y)\right) _{ij}\right) ^{-1}$.

\begin{rem}
\label{remrho} If $q$ belongs to the boundary $\partial M$, let $\bar{y}
=\left( y_{1},\dots ,y_{n-1}\right) $ be Riemannian normal coordinates on
the $n-1$ manifold $\partial M$ at the point $q$. For a point $\xi \in M$
close to $q$, there exists a unique $\bar{\xi}\in \partial M$ such that
$d_{g}(\xi ,\partial M)=d_{g}(\xi ,\bar{\xi})$. We set $\bar{y}(\xi )\in
\mathbb{R}^{n-1}$ the normal coordinates for $\bar{\xi}$ and $y_{n}(\xi
)=d_{g}(\xi ,\partial M)$. Then we define a chart $\psi _{q}^{\partial }:
\mathbb{R}_{+}^{n}\rightarrow M$ such that $\left( \bar{y}(\xi ),y_{n}(\xi
)\right) =\left( \psi _{q}^{\partial }\right) ^{-1}(\xi )$. These
coordinates are called \emph{Fermi coordinates} at $q\in \partial M$. The
Riemannian metric $g_{q}\left( \bar{y},y_{n}\right) $ read through the Fermi
coordinates satisfies $g_{q}(0) =\Id$.
\end{rem}

In the following we choose $\rho >0$ such that in the subset $\left(
\partial M\right) _{\rho }:=\left\{ x\in M\,:\,d_{g}(x,\partial M)<\rho
\right\} $ the Fermi coordinates are well defined. Moreover we choose $\rho $
small enough such that $3\rho $ is smaller than the radius 
$\rho (\partial M)$ of
topological invariance of $\partial M$, defined below.

\begin{defin}
The radius of topological invariance $\rho ({\cal M})$ of 
${\cal M}\subset \mathbb{R}^{N}$
is
\begin{equation*}
\rho ({\cal M}):=\sup \left\{ \rho >0\,:\,
\cat\left( ({\cal M})_{\rho }\right) =
\cat({\cal M}\right) \}
\end{equation*}
where
\begin{equation*}
\left({\cal M}\right)_{\rho }:=
\left\{ x\in \mathbb{R}^{N}\,:\,d(x,{\cal M})<\rho \right\}
\end{equation*}
\end{defin}

Fixed $\rho $, using Remark \ref{remR}, we can choose $R_{M}$ such that
$\cup _{i=1}^{l}B_{g}(x_{i},R_{M})$ covers $M\smallsetminus \left( \partial
M\right) _{\rho }$, and $R_{M}<\rho $. We note by $d_{g}^{\partial }$ and
$\exp ^{\partial }$ respectively the geodesic distance and the exponential
map on by $\partial M$. By compactness of $\partial M$, there is an
$R^{\partial }$ and a finite number of points $q_{i}\in \partial M$,
$i=1,\dots ,k$ such that
\begin{equation*}
I_{q_{i}}(R^{\partial },\rho ):=\left\{ x\in M,\,d_{g}(x,\partial M)=
d_{g}(x,\bar{\xi})<\rho ,\,d_{g}^{\partial }(q_{i},\bar{\xi})<R^{\partial }\right\}
\end{equation*}
form a covering of $\left( \partial M\right) _{\rho }$ and on every
$I_{q_{i}}$ the fermi coordinates are well defined. In the following we can
choose without loss of generality, $R=\min \left\{ R^{\partial
},R_{M}\right\} <\rho $.

\section{Main tools for the proof}

Using the notation of the previous section we can state our main result more
precisely.

\begin{teo}
\label{mainteo} There exists $\delta_0\in(0,m_e^+)$ and $\varepsilon_0>0$
such that, for $\delta\in(0,\delta_0)$ and $\varepsilon\in(0,\varepsilon_0)$
the functional $J_\varepsilon$ has at least $\cat(\partial M)$ critical
points $u\in\mathcal{N}_\varepsilon\subset H^1_g(M)$ satisfying
$J_\varepsilon(u)< m_e^++\delta$ and at least a critical point with $m_e^{+}
+ \delta\leq J_\varepsilon(u)\leq c$.
\end{teo}

We recall the definition of Lusternik Schnirelmann category.

\begin{defin}
Let $M$ a topological space and consider a closed subset $A\subset M$. We
say that $A$ has category $k$ relative to $M$ ($\cat_M A=k$) if $A$ is
covered by $k$ closed sets $A_j$, $j=1,\dots,k$, which are contractible in
$M $, and $k$ is the minimum integer with this property.
\end{defin}

\begin{rem}
Let $M_1$ and $M_2$ be topological spaces. If $g_1:M_1\rightarrow M_2$ and
$g_2:M_2\rightarrow M_1$ are continuous operators such that $g_2\circ g_1$ is
homotopic to the identity on $M_1$, then $\cat M_1\leq \cat M_2$.
For the proof see \cite{BC91}.
\end{rem}

We recall the following classical result (see for example \cite{BCP91}).

\begin{teo}
\label{teocat}Let $J$ be a $C^{1,1}$ real functional on a complete $C^{1,1}$
manifold $\mathcal{N}$. If $J$ is bounded from below and satisfies the
{\em Palais Smale}
condition then has at least $\cat(J^{d})$ critical point in $J^{d}$ where
$J^{d}=\{u\in \mathcal{N}\ :\ J(u)<d\}$. Moreover if $\mathcal{N}$ is
contractible and $\cat J^{d}>1$, there exists at least one critical point
$u\not\in J^{d}$
\end{teo}

Applying the first claim of Theorem \ref{teocat} to the functional
$J_{\varepsilon }$ on the manifold $\mathcal{N}_{\varepsilon }$ we obtain
$\cat\mathcal{N}_{\varepsilon }\cap J_{\varepsilon }^{m_{e}^{+}+\delta }$
critical points of $J_{\varepsilon }$. By the following Lemma we give an
estimate of $\cat\mathcal{N}_{\varepsilon }\cap J_{\varepsilon
}^{m_{e}^{+}+\delta}$ through the topological properties of the boundary of
$M$.

\begin{lemma}
For $\delta $ and $\varepsilon $ small enough we have $\cat(\partial M)\leq
\cat\mathcal{N}_{\varepsilon }\cap J_{\varepsilon }^{m_{e}^{+}+\delta }$.
\end{lemma}

We are able to obtain the proof of this lemma building two suitable maps. To
this aim we recall that by Nash embedding theorem \cite{Na56}
we may assume that $M$ is
embedded in a Euclidean space $\mathbb{R}^{N}$.

Hence the lemma follows by building a map $\Phi _{\varepsilon }:\partial
M\rightarrow \mathcal{N}_{\varepsilon }\cap
J_{\varepsilon}^{m_{e}^{+}+\delta }$ and a map
$\beta :\mathcal{N}_{\varepsilon }\cap J_{\varepsilon }^{m_{e}^{+}+\delta }
\rightarrow \left(
\partial M\right)_{\rho }$ with $0<\rho <\rho (\partial M)$ such that 
$\beta \circ
\Phi _{\varepsilon}:\partial M \rightarrow \left( \partial M\right) _{\rho }$
is homotopic to the identity on $\partial M$ (see sections 4,5,6).
Then by the properties of the
category we get $\cat(\partial M)\leq \cat\mathcal{N}_{\varepsilon }\cap
J_{\varepsilon }^{m_{e}^{+}+\delta }$.

To finish the proof of Theorem \ref{mainteo} we build a set
$T_{\varepsilon} $ (Section 7) such that
\begin{equation*}
\Phi _{\varepsilon }(\partial M)\subset T_{\varepsilon }\subset
\mathcal{N}_{\varepsilon }\cap J_{\varepsilon }^{c_{\varepsilon }}
\end{equation*}
for a bounded constant $c_{\varepsilon }\leq c$, and such that
$T_{\varepsilon }$ is a contractible set in
$\mathcal{N}_{\varepsilon }\cap J_{\varepsilon }^{c_{\varepsilon }}$ 
 containing  only positive functions. Since
$1<\cat(\partial M)\leq \cat\left(\Phi _{\varepsilon }(\partial M)\right)$ 
by the same
argument of Theorem \ref{teocat} there exists a critical point $\bar{u}$ of
$J_{\varepsilon }$ in $\mathcal{N}_{\varepsilon }$ such that
$m_{e}^{+}+\delta \leq J_{\varepsilon }\left( \bar{u}\right) \leq
c_{\varepsilon }$.

It remains to show that the critical points we have found are non-constant
functions. This follows immediately from the fact that the only constant
function on the Nehari manifold $\mathcal{N}_{\varepsilon }$ is the function
$\bar{v}(x)\equiv 1$, for which
\begin{equation*}
J_{\varepsilon }(\bar{v})=\left( \frac{1}{2}-\frac{1}{p}\right) \frac{\mu
_{g}(M)}{\varepsilon^n }\rightarrow \infty \text{ as }\varepsilon \rightarrow
0.
\end{equation*}
Hence the constant solution is excluded because $c_{\varepsilon }$ is
bounded.

\subsection{Notation}

We will use the following notation

\begin{itemize}
\item $\displaystyle||u||_{g}=||u||_{H_{g}^{1}}=\int_{M}|\nabla
_{g}u|^{2}+|u|^{2}d\mu _{g}$,\ \
$\displaystyle |u|_{p,g}^{p}=\int_{M}|u|^{p}d\mu _{g};$

\item $\displaystyle|||u|||_{\varepsilon }=
\displaystyle|||u|||_{\varepsilon,M }=\frac{1}{\varepsilon ^{n}}
\int_{M}\varepsilon ^{2}|\nabla _{g}u|^{2}+|u|^{2}d\mu _{g}$,\ \
$\displaystyle |u|_{p,\varepsilon }^{p}=\frac{1}{\varepsilon ^{n}}\int_{M}|u|^{p}d\mu _{g};$

\item $\displaystyle|u|_{p}^{p}=\int_{\mathbb{R}^{n}}|u|^{p}dx;$

\item If $A,B\subset \mathbb{R}^{n}$, then $A\Delta B:=A\smallsetminus B\cup
B\smallsetminus A.$

\item $d_{g}$ is the geodesic distance on $M$, and $d_{g}^{\partial }$ is
the geodesic distance on $\partial M$.

\item $\exp ^{\partial }$ is the exponential map on $\partial M$.

\item $I_q(R,\rho)=\{\chi\in M\ :\ d_g(\chi,\partial M)<\rho, \ d_g^\partial(\bar{\chi},q)<R\}$,
where $\bar{\chi}\in\partial M$ is the unique point such
that $d_g(\chi,\bar{\chi})= d_g(\chi,\partial M)$.

\item $B(x,R)\subset \mathbb{R}^n$ is the ball centered in $x$ of radius $R$.

\item $B_{n-1}(x,R)\subset \mathbb{R}^{n-1}$ is the $n-1$ ball centered in
$x $ of radius $R$.
\end{itemize}

\section{The map $\Phi _{\protect\varepsilon }$}

Let us define $\chi _{R}:\mathbb{R^{+}}\rightarrow \mathbb{R^{+}}$ a smooth
cut off function such that $\chi _{R}(t)\equiv 1$ if $0\leq t\leq R/2$,
$\chi_{R}(t)\equiv 0$ if $R\leq t$, and
$|\chi _{R}^{\prime }(t)|\leq 2/R$ for all
$t$. Fixed a point $q\in \partial M$ and $\varepsilon >0$, let us define on
$M $ the function $Z_{\varepsilon ,q}(\xi )$ as
\begin{equation}\label{zeq}
Z_{\varepsilon ,q}(\xi )=\left\{
\begin{array}{cl}
V_{\varepsilon }\left( y(\xi )\right) \chi _{R}\left( |\bar{y}(\xi )|\right)
\chi _{\rho }\left( y_{n}(\xi )\right) & \text{if }\xi \in I_{q} \\
&  \\
0 & \text{otherwise}
\end{array}
\right.
\end{equation}
where
\begin{equation*}
I_{q}(R,\rho )=I_{q}=\left\{ \xi \in M:y_{n}=d_{g}(\xi ,\partial M)<\rho
\text{ and }\left\vert \bar{y}\right\vert =d_{g}^{\partial }\left( \exp
_{q}^{\partial }(\bar{y}(\xi )),q\right) <R\right\} .
\end{equation*}
Here $y(\xi )=(\bar{y}(\xi ),y_{n}(\xi ))=\left( \psi _{q}^{\partial
}\right) ^{-1}(\xi )$ are the Fermi coordinates at $q\in \partial M$ and
$\exp _{q}^{\partial }:T_{q}(\partial M)\rightarrow \partial M$, is the
exponential map on $\partial M$.

For each $\varepsilon >0$ we can define a positive number $t_{\varepsilon
}(Z_{\varepsilon ,q})$ such that $t_{\varepsilon }(Z_{\varepsilon
,q})Z_{\varepsilon ,q}\in H_{g}^{1}(M)\cap \mathcal{N_{\varepsilon }}$.
Namely, $t_{\varepsilon }(Z_{\varepsilon ,q})$ turns out to verify
\begin{equation}
t_{\varepsilon }(Z_{\varepsilon ,q})=
\left( \frac{|||Z_{\varepsilon,q}|||_{\varepsilon }^{2}}
{|Z_{\varepsilon ,q}|_{p,\varepsilon}^{p}}\right) ^{\frac{1}{{p-2}}}.
\end{equation}
Thus we can define a function $\Phi _{\varepsilon }:\partial M\rightarrow
\mathcal{N_{\varepsilon }}$, $\Phi _{\varepsilon }(q)=t_{\varepsilon
}(Z_{\varepsilon ,q})Z_{\varepsilon ,q}$

\begin{prop}
\label{propphi}For any $\varepsilon >0$ the application $\Phi _{\varepsilon
}:\partial M\rightarrow \mathcal{N}_{\varepsilon }$ is continuous. Moreover,
for any $\delta >0$ there exists $\varepsilon _{0}=\varepsilon _{0}(\delta
)>0$ such that, if $\varepsilon <\varepsilon _{0}$ then
\begin{equation*}
\Phi _{\varepsilon }(q)\in \mathcal{N}_{\varepsilon }\cap J_{\varepsilon
}^{m_{e}^{+}+\delta }\text{ for all }q\in \partial M
\end{equation*}
\end{prop}

\begin{proof}
Fixed $\varepsilon >0$, by the continuity of $u\rightarrow
t_{\varepsilon }(u)$ on $H_{g}^{1}(M)$ it is enough to prove that for any
sequence $\left\{ q_{k}\right\} \subset \partial M$ convergent to $q$ we
have
\begin{equation*}
\lim_{k\rightarrow \infty }\left\vert \left\vert Z_{\varepsilon
,q_{k}}-Z_{\varepsilon ,q}\right\vert \right\vert _{H_{g}^{1}}=0.
\end{equation*}
Since $q_{k}$ converges to $q$, we have $\mu _{g}(I_{q_{k}}\Delta
I_{q})\rightarrow 0$ as $k\rightarrow \infty $, then we have
\begin{equation*}
\int_{I_{q_{k}}\Delta I_{q}}\left\vert Z_{\varepsilon ,q_{k}}-Z_{\varepsilon
,q}\right\vert ^{2}d\mu _{g}\rightarrow 0\text{ as }k\rightarrow \infty .
\end{equation*}
Now, setting $\eta _{k}(\bar{y},y_{n})=\left( \psi _{q_{k}}^{\partial
}\right) ^{-1}\left( \psi _{q}^{\partial }(\bar{y},y_{n})\right) $ and
$A_{k}=\left( \psi _{q}^{\partial }\right) ^{-1}(I_{q_{k}}\cap I_{q})$ we can
write

\begin{multline*}
\int_{I_{q_{k}}\cap I_{q}} \left\vert
Z_{\varepsilon,q_{k}}(x)-Z_{\varepsilon ,q}(x)\right\vert ^{2}d\mu _{g}= \\
\int_{A_{k}}\Big\vert V_{\varepsilon }(\eta _{k}(\bar{y},y_{n}))
\chi_{R}\left( \left\vert \pi _{\mathbb{R}^{n-1}}\eta _{k}(\bar{y},y_{n})
\right\vert \right) \chi _{\rho }\left( d_{g}(q_{k},\partial M)\right) - \\
-V_{\varepsilon }\left( \bar{y},y_{n}\right) \chi _{R} \left(\left\vert
\bar{y}\right\vert \right) \chi _{\rho }\left( d_{g}(q,\partial M)\right)
\Big\vert ^{2} \left\vert g_{q}(\bar{y},y_{n})\right\vert ^{1/2}d\bar{y}
dy_{n}\leq \\
\leq c\int_{A_{k}}\left\vert \eta _{k}(\bar{y},y_{n})-
(\bar{y},y_{n})\right\vert ^{2}d\bar{y}dy_{n}
\end{multline*}
for a suitable constant $c$ coming from the mean value theorem applied to
$V_{\varepsilon },\chi _{\rho },\chi _{R}$. By the definition of $\eta _{k}$
and the smoothness of the exponential map we get
\begin{equation*}
\left\vert \left\vert Z_{\varepsilon ,q_{k}}-Z_{\varepsilon ,q}\right\vert
\right\vert _{L_{g}^{2}}\rightarrow 0\text{ as }k\rightarrow \infty .
\end{equation*}
A similar argument can be used to show that $\left\vert \left\vert \nabla
_{g}Z_{\varepsilon ,q_{k}}-\nabla _{g}Z_{\varepsilon ,q}\right\vert
\right\vert _{L_{g}^{2}}\rightarrow 0$ as $k\rightarrow \infty $.

To prove the second statement of the theorem we first show that the
following limits hold uniformly with respect to $q\in \partial M$.
\begin{equation}
\lim_{\varepsilon \rightarrow 0}\left\vert \left\vert Z_{\varepsilon
,q}\right\vert \right\vert _{2,\varepsilon }^{2}=
\int_{\mathbb{R}_{+}^{n}}V^{2}(y)dy  \label{eql2}
\end{equation}
\begin{equation}
\lim_{\varepsilon \rightarrow 0}\left\vert \left\vert Z_{\varepsilon
,q}\right\vert \right\vert _{p,\varepsilon }^{p}=
\int_{\mathbb{R}_{+}^{n}}V^{p}(y)dy  \label{eqlp}
\end{equation}
\begin{equation}
\lim_{\varepsilon \rightarrow 0}\varepsilon ^{2}\left\vert \left\vert \nabla
Z_{\varepsilon ,q}\right\vert \right\vert _{2,\varepsilon }^{2}=
\int_{\mathbb{R}_{+}^{n}}\left\vert \nabla V\right\vert ^{2}(y)dy  \label{eqgrad}
\end{equation}
where $\left\vert \left\vert u\right\vert \right\vert _{q,\varepsilon }=
\frac{1}{\varepsilon ^{n}}\left\vert \left\vert u\right\vert \right\vert
_{L^{q}}$. For (\ref{eql2}) we have
\begin{eqnarray*}
&&\frac{1}{\varepsilon ^{n}} \int_{M}\left\vert
Z_{\varepsilon,q}(x)\right\vert ^{2}d\mu _{g}= \\
&&=\frac{1}{\varepsilon ^{n}} \int_{|\bar{y}|<R,0<y_{n}<\rho}
V_{\varepsilon}^{2}(\bar{y},y_{n})\chi _{R}^{2} 
\left( \left\vert \bar{y}\right\vert\right) 
\chi _{\rho }^{2} \left( y_{n}\right) \left\vert
g_q(\bar{y},y_{n})\right\vert ^{1/2}d\bar{y}dy_{n}= \\
&&=\int_{|\bar{z}|<R/\varepsilon ,0<z_{n}<\rho /\varepsilon }V^{2}(\bar{z},z_{n})
\chi _{R/\varepsilon }^{2}\left( \left\vert \bar{z}\right\vert\right)
\chi _{\rho /\varepsilon }^{2}\left( z_{n}\right) \left\vert
g_q(\varepsilon (\bar{z},z_{n}))\right\vert ^{1/2}
d\bar{z}dz_{n}= \\
&&=\int_{B_K}V^{2}(\bar{z},z_{n})\chi _{R/\varepsilon }^{2}
\left(\left\vert \bar{z}\right\vert \right) 
\chi _{\rho /\varepsilon }^{2}
\left(z_{n}\right) 
\left\vert g_q(\varepsilon (\bar{z},z_{n}))\right\vert ^{1/2}d\bar{z}dz_{n}+ \\
&&+\int_{\mathbb{R}^{n}_+\smallsetminus B_K}V^{2}(\bar{z},z_{n})
\chi_{R/\varepsilon }^{2}\left( \left\vert \bar{z}\right\vert \right)
\chi_{\rho /\varepsilon }^{2}\left( z_{n}\right) \left\vert
g_q(\varepsilon (\bar{z},z_{n}))\right\vert ^{1/2}
d\bar{z}dz_{n},
\end{eqnarray*}
where $B_k=B(0,K)\cap \{z_n>0\}$.
It is easy to see that the second addendum vanishes when
$K\rightarrow\infty$. With respect to the first addendum, fixed $K$ large
enough, by compactness of manifold $M$ and regularity of the exponential
map and of the Riemannian metric $g$ we have,
for $\varepsilon\rightarrow0$,
\begin{equation*}
\int_{B_K}V^{2}(\bar{z},z_{n})\chi _{R/\varepsilon }^{2}\left( \left\vert
\bar{z}\right\vert \right) \chi _{\rho /\varepsilon }^{2}\left( z_{n}\right)
\left\vert g_{\psi _{q}^{\partial }}(\varepsilon (\bar{z},z_{n}))\right\vert
^{1/2}d\bar{z}dz_{n}\rightarrow \int_{B_K}V^{2}(y)dy
\end{equation*}
uniformly  with respect to $q\in \partial M$. So we prove (\ref{eql2}). In
the same way we can prove (\ref{eqlp}) and (\ref{eqgrad}).

At this point we observe that
\begin{equation*}
J_{\varepsilon }(t_{\varepsilon }(Z_{\varepsilon ,q})Z_{\varepsilon
,q})=\left( \frac{1}{2}-\frac{1}{p}\right) [t_{\varepsilon }(Z_{\varepsilon
,q})]^{p}\left\vert \left\vert Z_{\varepsilon ,q}\right\vert \right\vert
_{\varepsilon ,p}^{p}.
\end{equation*}
By definition of $t_{\varepsilon }(Z_{\varepsilon ,q})$ and by (\ref{eql2}),
(\ref{eqlp}) and (\ref{eqgrad}) we have that $t_{\varepsilon
}(Z_{\varepsilon ,q})\rightarrow 1$ as $\varepsilon \rightarrow 0$,
uniformly with respect to $q\in \partial M$. Concluding we have
\begin{equation}  \label{jeps}
\lim_{\varepsilon \rightarrow 0}J_{\varepsilon }(t_{\varepsilon
}(Z_{\varepsilon ,q})Z_{\varepsilon ,q})=\left( \frac{1}{2}-\frac{1}{p}
\right) \int_{\mathbb{R}_{+}^{n}}V^{p}(y)dy=m_{e}^{+}
\end{equation}
uniformly with respect to $q\in \partial M$.
\end{proof}

\begin{rem}
\label{remlimsup}By Proposition \ref{propphi}, given $\delta $, we have that
$\mathcal{N}_{\varepsilon }\cap J_{\varepsilon }^{m_{e}^{+}+\delta }
\neq\emptyset$ for $\varepsilon $ small enough. Moreover let
\begin{equation*}
m_{\varepsilon }:=\inf \left\{ J_{\varepsilon }(u)\,:\,
u\in \mathcal{N}_{\varepsilon }\right\} .
\end{equation*}
At this point we have
\begin{equation*}
\limsup_{\varepsilon \rightarrow 0}m_{\varepsilon }\leq m_{e}^{+}.
\end{equation*}
\end{rem}

\section{Concentration properties}

In this section we will show a property of concentration of the functions
$u\in \mathcal{N}_{\varepsilon }\cap J_{\varepsilon }^{m_{e}^++\delta }$ when
$\varepsilon $ and $\delta $ are sufficiently small. This concentration
property will be crucial to verify that the barycenter $\beta (u)$ (see
Section 6) of the functions $u\in \mathcal{N}_{\varepsilon }\cap
J_{\varepsilon }^{m_{e}^++\delta }$ is close to the boundary $\partial M$.

For any $\varepsilon >0$ we can construct a finite closed partition
$\mathcal{P}^{\varepsilon }=\left\{ P_{j}^{\varepsilon }\right\}_{j\in
\Lambda_\varepsilon}$ of $M$ such that

\begin{itemize}
\item $P_{j}^{\varepsilon }$ is closed for every $j$;

\item $P_{j}^{\varepsilon }\cap P_{k}^{\varepsilon }\subset \partial
P_{j}^{\varepsilon }\cap \partial P_{k}^{\varepsilon }$ for $j\neq k$;

\item $K_1\varepsilon \leq d_j^{\varepsilon }\leq K_2\varepsilon $, where
$d_j^{\varepsilon }$ is the diameter of $P_{j}^{\varepsilon }$;

\item $c_1\varepsilon ^{n}\leq \mu _{g}\left( P_{j}^{\varepsilon }\right)
\leq c_2\varepsilon ^{n}$;

\item for any $j$ there exists an open set $I_{j}^{\varepsilon }\supset
P_{j}^{\varepsilon }$ such that, if $P_{j}^{\varepsilon }\cap \partial
M=\emptyset $, then $d_{g}\left( I_{j}^{\varepsilon },\partial M\right)
>K\varepsilon /2$, while, if $P_{j}^{\varepsilon }\cap \partial M\neq
\emptyset $, then $I_{j}^{\varepsilon }\subset \left\{ x\in
M\,:\,d_{g}\left( x,\partial M\right) \leq \frac{3}{2}K\varepsilon \right\}$;

\item there exists a finite number $\nu (M)\in \mathbb{N}$ such that every
$x\in M$ is contained in at most $\nu (M)$ sets $I_{j}^{\varepsilon }$, where
$\nu (M)$ does not depends on $\varepsilon $.
\end{itemize}

By compactness of $M$ such a partition exists, at least for small
$\varepsilon $. In the following we will choose always 
$\varepsilon_0(\delta)$
sufficiently small in order to have this partition.

\begin{lemma}
\label{lemmagamma}There exists a constant $\gamma >0$ such that, for any
fixed $\delta >0$ and for any $\varepsilon \in (0,\varepsilon _{0}(\delta ))$,
where $\varepsilon _{0}(\delta )$ is as in Proposition \ref{propphi},
given any partition $\mathcal{P}^{\varepsilon }$of $M$ as above, and any
function $u\in \mathcal{N}_{\varepsilon }\cap J_{\varepsilon }^{m_{e}^++\delta
}$, there exists a set $P_{j}^{\varepsilon }\subset \mathcal{P}^{\varepsilon
}$ such that
\begin{equation*}
\frac{1}{\varepsilon ^{n}}\int_{P_{j}^{\varepsilon }}|u^{+}|^{p}d\mu
_{g}\geq \gamma >0.
\end{equation*}
\end{lemma}

\begin{proof}
By Remark \ref{remlimsup} we have that 
$\mathcal{N}_{\varepsilon }\cap J_{\varepsilon }^{m_{e}^++\delta}\neq\emptyset$.
For any function 
$u\in \mathcal{N}_{\varepsilon }\cap J_{\varepsilon }^{m_{e}^++\delta}$ we denote
by $u_{j}^{+}$ the restriction of $u^{+}$ to the set $P_{j}^{\varepsilon }$.
Then we can write
\begin{multline*}
\frac{1}{\varepsilon ^{n}}\int_{M}\left( \varepsilon ^{2}|\nabla
_{g}u|^{2}+u^{2}\right) d\mu _{g} =\frac{1}{\varepsilon ^{n}}
\int_{M}(u^{+})^{p}d\mu _{g}=\frac{1}{\varepsilon ^{n}}\sum_{j}
\int_{M}(u_{j}^{+})^{p}d\mu _{g}= \\
=\sum_{j}\frac{|u_{j}^{+}|_{p}^{p-2}}{\varepsilon ^{\frac{n(p-2)}{p}}}
\frac{|u_{j}^{+}|_{p}^{2}}{\varepsilon ^{\frac{2n}{p}}}\leq \max_{j}\left\{
\frac{|u_{j}^{+}|_{p}^{p-2}}{\varepsilon ^{\frac{n(p-2)}{p}}}\right\}
\sum_{j}\frac{|u_{j}^{+}|_{p}^{2}}{\varepsilon ^{\frac{2n}{p}}}.
\end{multline*}
We define the functions $\tilde{u}_{j}$ by using a smooth real cutoff
function $\chi _{\varepsilon }^{j}:M\rightarrow [0,1]$ such that
$|\nabla _{g}\chi _{\varepsilon }^{j}|\leq \frac{K}{\varepsilon }$ for some
constant $K$ and, if $P_{j}^{\varepsilon }\cap \partial M=\emptyset $, then
$\chi _{\varepsilon }^{j}=1$ for $x\in P_{j}^{\varepsilon }$ and $\chi
_{\varepsilon }^{j}=0$ for $x\in M\smallsetminus I_{j}^{\varepsilon }$,
while if $P_{j}^{\varepsilon }\cap \partial M\neq \emptyset $, then
$\chi_{\varepsilon }^{j}=1$ for $x\in P_{j}^{\varepsilon }$ and
$\chi_{\varepsilon }^{j}=0$ for $M\smallsetminus \bar{I}_{j}^{\varepsilon }$ and
$x\in \partial I_{j}^{\varepsilon }\cap (M\smallsetminus \partial M)$. So we
define
\begin{equation*}
\tilde{u}_{j}(x)=u^{+}(x)\chi _{\varepsilon }^{j}(x).
\end{equation*}
It holds $\tilde{u}_{j}\in H_{g}^{1}(M)$, hence using Sobolev inequalities
there exists a positive constant $C$ such that, for any $j$,
\begin{equation*}
\frac{|u_{j}^{+}|_{p}^{2}}{\varepsilon ^{\frac{2n}{p}}}\leq
\frac{|\tilde{u}_{j}|_{p}^{2}}{\varepsilon ^{\frac{2n}{p}}}\leq
C|||\tilde{u}_{j}|||_{\varepsilon}^{2}=
C|||\tilde{u}_{j}|||_{\varepsilon,P_{j}^{\varepsilon } }^{2}+
C|||\tilde{u}_{j}|||_{\varepsilon,I_{j}^{\varepsilon }\smallsetminus P_{j}^{\varepsilon }}^{2}.
\end{equation*}
Moreover
\begin{eqnarray*}
\int_{I_{j}^{\varepsilon }\smallsetminus P_{j}^{\varepsilon }}\left\vert
\tilde{u}_{j}\right\vert ^{2}d\mu _{g} &\leq &\int_{I_{j}^{\varepsilon
}\smallsetminus P_{j}^{\varepsilon }}\left\vert u^{+}\right\vert ^{2}d\mu
_{g}; \\
\int_{I_{j}^{\varepsilon }\smallsetminus P_{j}^{\varepsilon }}\varepsilon
^{2}\left\vert \nabla \tilde{u}_{j}\right\vert ^{2}d\mu _{g} &\leq
&\int_{I_{j}^{\varepsilon }\smallsetminus P_{j}^{\varepsilon }}(\varepsilon
^{2}\left\vert \nabla u^{+}\right\vert ^{2}+K^{2}\left\vert u^{+}\right\vert
^{2})d\mu _{g}.
\end{eqnarray*}
Hence we obtain
\begin{eqnarray*}
\sum_{j}\frac{|u_{j}^{+}|_{p}^{2}}{\varepsilon ^{\frac{2n}{p}}} &\leq
&C\sum_{j}\left\vert \left\vert \left\vert u^{+}\right\vert \right\vert
\right\vert _{\varepsilon }^{2}+C(K^{2}+1)\nu (M)\left\vert \left\vert
\left\vert u^{+}\right\vert \right\vert \right\vert _{\varepsilon }^{2}\leq
\\
&\leq &C(K^{2}+2)\nu (M)\frac{1}{\varepsilon ^{n}}\int_{M}(\varepsilon
^{2}\left\vert \nabla u\right\vert ^{2}+\left\vert u\right\vert ^{2})d\mu
_{g}.
\end{eqnarray*}
We can conclude that
\begin{equation*}
\max_{j}\left\{ \left( \frac{1}{\varepsilon ^{n}}\int_{P_{j}^{\varepsilon
}}\left\vert u^{+}\right\vert ^{p}d\mu _{g}\right) ^{\frac{p-2}{p}}\right\}
\geq \frac{1}{C(K^{2}+2)\nu (M)},
\end{equation*}
so the proof is complete.
\end{proof}

\begin{rem}
\label{ekeland}Let $\delta $ and $\varepsilon $ fixed. For any $u\in
\mathcal{N}_{\varepsilon }\cap J_{\varepsilon }^{m_{\varepsilon}+2\delta }$ there
exists $u_{\delta }\in \mathcal{N}_{\varepsilon }$ such that
\begin{equation*}
J_{\varepsilon }(u_{\delta })<J_{\varepsilon }(u),~\left\vert \left\vert
\left\vert u_{\delta }-u\right\vert \right\vert \right\vert _{\varepsilon }<4
\sqrt{\delta };
\end{equation*}
\begin{equation*}
\left\vert \left({J_\varepsilon}_{|\mathcal{N}_{\varepsilon }}\right)'
(u_{\delta})[\xi ]\right\vert <
\sqrt{\delta }\left\vert \left\vert \left\vert \xi
\right\vert \right\vert \right\vert _{\varepsilon }.
\end{equation*}

This is simply the application of Ekeland variational principle
(see \cite{def89}) to the
functional $J_{\varepsilon }$ on the manifold $\mathcal{N}_{\varepsilon }$.
\end{rem}

\begin{prop}
\label{propconc}For all $\eta \in (0,1)$ there exists a $\delta
_{0}<m_{e}^{+}$ such that for any $\delta \in (0,\delta _{0})$ for any
$\varepsilon \in (0,\varepsilon _{0}(\delta ))$ (as in Prop. \ref{propphi})
and for any function $u\in \mathcal{N}_{\varepsilon }\cap
J_{\varepsilon}^{m_{e}^{+}+\delta }$ we can find a point $q=q(u)\in \partial M$ for which
\begin{equation*}
\left( \frac{1}{2}-\frac{1}{p}\right) \frac{1}{\varepsilon ^{n}}
\int_{I_{q}(\rho ,R)}|u^{+}|^{p}d\mu _{g}\geq (1-\eta )m_{e}^{+}
\end{equation*}
where $I_{q}(\rho ,R)$ is defined in the notation paragraph.
\end{prop}

\begin{proof}
We prove this property for $u\in \mathcal{N}_{\varepsilon }\cap
J_{\varepsilon}^{m_{e}^{+}+\delta }\cap 
J_{\varepsilon}^{m_\varepsilon+2\delta }$. From the thesis for these functions 
follows that 
\begin{equation}\label{eqgeq}
m_\varepsilon\geq (1-\eta)m_e^+.
\end{equation}
By (\ref{eqgeq}) and by Remark \ref{remlimsup} we have that 
\begin{equation}
\lim_{\varepsilon\rightarrow0}m_\varepsilon=m_e^+.
\end{equation}
Thus $J_{\varepsilon}^{m_{e}^{+}+\delta }\subset 
J_{\varepsilon}^{m_\varepsilon+2\delta }$ for $\varepsilon,\delta$ 
small enough, and the general case is proved.

The proof is by contradiction. Hence we assume that there
exists $\eta \in (0,1)$, two sequences of vanishing real numbers $\left\{
\delta _{k}\right\} _{k}$ and $\left\{ \varepsilon _{k}\right\} _{k}$ and a
sequence of functions $\left\{ u_{k}\right\} _{k}\subset
\mathcal{N}_{\varepsilon _{k}}\cap 
J_{\varepsilon _{k}}^{m_{e}^++\delta _{k}}
\cap J_{\varepsilon _{k}}^{m_{\varepsilon_k}+2\delta _{k}}$
such that,
for any $q\in \partial M$ it holds
\begin{equation}
\left( \frac{1}{2}-\frac{1}{p}\right) \frac{1}{\varepsilon _{k}^{n}}
\int_{I_{q}(\rho ,R)}|u_{k}^{+}|^{p}d\mu _{g}<(1-\eta )m_{e}^{+}.  \label{uk}
\end{equation}
By Remark \ref{ekeland} and by definition of $\mathcal{N}_{\varepsilon _{k}}$
we can assume
\begin{equation*}
J_{\varepsilon _{k}}^{\prime }(u_{k})[\varphi ]\leq \sqrt{\delta _{k}}
\left\vert \left\vert \left\vert \varphi \right\vert \right\vert \right\vert
_{\varepsilon _{k}}\text{ for all }\varphi \in H_{g}^{1}(M).
\end{equation*}
By Lemma \ref{lemmagamma} there exists a set $P_{k}^{\varepsilon _{k}}\in
\mathcal{P}_{\varepsilon _{k}}$ such that
\begin{equation*}
\frac{1}{\varepsilon _{k}^{n}}\int_{P_{k}^{\varepsilon
_{k}}}|u_{k}^{+}|^{p}d\mu _{g}\geq \gamma >0.
\end{equation*}
we have to examine two cases: either there exists a subsequence
$P_{i_{k}}^{\varepsilon _{i_{k}}}$ such that $P_{i_{k}}^{\varepsilon
_{i_{k}}}\cap \partial M\neq \emptyset $, or there exists a subsequence
$P_{i_{k}}^{\varepsilon _{i_{k}}}$ such that $P_{i_{k}}^{\varepsilon
_{i_{k}}}\cap \partial M=\emptyset $. For simplicity we write simply $P_{k}$
for $P_{i_{k}}^{\varepsilon _{i_{k}}}$.

\textbf{The case}$\ P_{k}\cap \partial M\neq \emptyset $. We choose a point
$q_{k} $ interior to $P_{k}\cap \partial M$. We have the Fermi coordinates
$\psi _{q_{k}}^{\partial }:B_{n-1}(0,R)\times [ 0,\rho ]\rightarrow M$,
$\psi _{q_{k}}^{\partial }(\bar{y},y_{n})=(\bar{x},x_{n})=x$. We consider
the function $w_{k}:\mathbb{R}_{+}^{n}\rightarrow \mathbb{R}$ defined by

\begin{equation*}
u_{k}(\psi _{q_{k}}^{\partial }(\bar{y},y_{n}))\chi _{R}(|\bar{y}|)\chi
_{\rho }(y_{n})=u_{k}(\psi _{q_{k}}^{\partial }
(\varepsilon _{k}\bar{z},\varepsilon z_{n}))\chi _{R}(|\varepsilon _{k}\bar{z}|)
\chi _{\rho}(\varepsilon z_{n})=w_{k}(\bar{z},z_{n}).
\end{equation*}
It is clear that $w_{k}\in H^{1}(\mathbb{R}_{+}^{n})$ with
$w_{k}(\bar{z},z_{n})=0$ when $|\bar{z}|=0,R/\varepsilon _{k}$ or
$z_{n}=\rho /\varepsilon_{k}$. We now show some properties of the function $w_{k}$.

\textsc{Step1:} {\em There exists a $w\in H^{1}(\mathbb{R}_{+}^{n})$
such that the sequence $w_{k}$ converges weakly in
$H^{1}(\mathbb{R}_{+}^{n})$ and strongly in
$L_{\text{loc}}^{p}(\mathbb{R}_{+}^{n})$}

We have the following inequality
\begin{eqnarray}
&& \frac{1}{\varepsilon _{k}^{n}}\int_{M}|u_{k}|^{2}d\mu _{g} \geq  \notag \\
&&\geq \frac{1}{\varepsilon _{k}^{n}}\int_{B_{n-1}(0,R)\times [ 0,\rho
]}|u_{k}\left( \psi _{q_{k}}^{\partial }(y)\right) |^{2}\chi _{R}^{2}
\left( |\bar{y}|\right) \chi _{\rho }^{2}\left( (y_{n})\right) \left\vert
g_{q_{k}}(y)\right\vert ^{1/2}dy=  \label{equk1} \\
&&=\int_{B_{n-1}(0,R/\varepsilon _{k})\times [ 0,\rho /\varepsilon
_{k}]}|w_{k}|^{2}\left\vert g_{q_{k}}(\varepsilon z)\right\vert ^{1/2}dz\geq
c|w_{k}|_{L^{2}(\mathbb{R}_{+}^{n})}^{2}.  \notag
\end{eqnarray}
Where $z=\varepsilon y$ and $c>0$ is a
suitable constant.

For simplicity we set $\tilde{\chi}(y)=\chi _{R}(\bar{y})\chi _{\rho
}(y_{n}) $ We have
\begin{eqnarray*}
&&\int_{\mathbb{R}_{+}^{n}}\left\vert \nabla w_{k}\right\vert ^{2}dx \leq \\
&\leq &2\int_{\mathbb{R}_{+}^{n}}\sum_{i}\left( \frac{\partial u_{k}}
{\partial z_{i}}(\varepsilon _{k}z)\right) ^{2}\tilde{\chi}^{2}(\varepsilon
_{k}z)dz+2\int_{\mathbb{R}_{+}^{n}}\sum_{i}u_{k}^{2}(\varepsilon
_{k}z)\left( \frac{\partial \tilde{\chi}}{\partial z_{i}}(\varepsilon
_{k}z)\right) ^{2}dz \\
&=&I_{1}+I_{2}
\end{eqnarray*}
By definition of $\tilde{\chi}$ and $w_{k}$ we have

\begin{eqnarray}
&&\frac{\varepsilon _{k}^{2}}{\varepsilon _{k}^{n}}\int_{M}\left\vert \nabla
_{g}u_{k}\right\vert ^{2}d\mu _{g}\geq \frac{\varepsilon _{k}^{2}}
{\varepsilon _{k}^{n}}\int_{\psi _{q_{k}}^{\partial }\left(
B_{n-1}(0,R)\times [ 0,\rho ]\right) }\left\vert \nabla
_{g}u_{k}\right\vert ^{2}d\mu _{g}=  \label{equk2} \\
&=&\int_{B_{n-1}(0,R/\varepsilon _{k})
\times [ 0,\rho /\varepsilon_{k}]}\sum_{ij}g_{q_{k}}^{ij}
\frac{\partial u_{k}}{\partial z_{i}}
(\varepsilon _{k}z)\frac{\partial u_{k}}{\partial z_{j}}
(\varepsilon_{k}z)\left\vert g_{q_{k}}
(\varepsilon z)\right\vert ^{1/2}dz\geq cI_{1}.
\notag
\end{eqnarray}
where $c$ depends only on the Riemannian manifold $M$. In a similar way we
have
\begin{equation}
I_{2}\leq \frac{c\varepsilon _{k}^{2}}{R^{2}\rho ^{2}\varepsilon _{k}^{n}}
\int_{M}\left\vert u_{k}\right\vert ^{2}d\mu _{g}.  \label{equk3}
\end{equation}
By (\ref{equk1}), (\ref{equk2}) and (\ref{equk3}) we get that $\left\vert
\left\vert w_{k}\right\vert \right\vert _{H^{1}(\mathbb{R}_{+}^{n})}$ is
bounded. Then we have the claim.

\textsc{Step2:} {\em The limit function $w$ is a weak solution of}
\begin{equation*}
\left\{
\begin{array}{cc}
-\Delta w+w=(w^{+})^{p-1} & \text{in }\mathbb{R}_{+}^{n}; \\
\frac{\partial w}{\partial \nu }=0 & \text{for }y=(\bar{y},0);
\end{array}
\right.
\end{equation*}

Firstly for any $\varphi \in C_{0}^{\infty }(\mathbb{R}_{+}^{n})$ we define on
the manifold $M$ the function $\tilde{\varphi}_{k}(x):=
\varphi \left( \frac{1}{\varepsilon _{k}}
\left( \psi _{q_{k}}^{\partial }\right)^{-1}(x)\right) $. We have that
\begin{eqnarray}
\nonumber
|||\tilde{\varphi}_{k}|||_{\varepsilon _{k}}&=&
\int_{\mathbb{R}_{+}^{n}}
\left[\sum_{ij}g_{q_{k}}^{ij}(\varepsilon _{k}z)
\frac{\partial \varphi}{\partial z_{i}}(z)
\frac{\partial \varphi}{\partial z_{j}}(z)+
|\varphi(z)|^{2}\right] |g_{q_{k}}(\varepsilon _{k}z)|^{1/2}dz\\
&\leq& c||\varphi||_{H^{1}}^{2}(\R^n_+)\label{normaphik}
\end{eqnarray}
where $c$ depends only on $M$.

We set
\begin{equation*}
F_{\varepsilon _{k}}(v)=\int_{\R^n_+}\left[
\sum_{ij}\frac{g_{q_{k}}^{ij}(\varepsilon _{k}z)}{2}
\frac{\partial v}{\partial z_{i}}
(z)\frac{\partial v}{\partial z_j}(z)+\frac{v^{2}(z)}{2}-
\frac{|w_{k}^{+}(z)|^{p}}{p}\right] |g_{q_{k}}(\varepsilon_{k}z)|^{1/2}dz
\end{equation*}
so
\begin{eqnarray*}
&&|F_{\varepsilon _{k}}^{\prime }(w_{k})[\varphi ]|=\\
&&=\int_{\text{supp}\varphi }
\left[ \sum_{ij}g_{q_{k}}^{ij}(\varepsilon _{k}z)\frac{\partial w_{k}}
{\partial z_{i}}(z)\frac{\partial \varphi }{\partial z_j}(z)+
\left(w_{k}(z)-(w_{k}^{+}(z))^{p-1}\right)\varphi (z)\right] \left\vert
g_{q_{k}}(\varepsilon _{k}z)\right\vert ^{1/2}.
\end{eqnarray*}
It is easy to verify that for $k=k(\varphi )$ large enough
\begin{equation*}
|F_{\varepsilon _{k}}^{\prime }(w_{k})[\varphi ]|=|J_{\varepsilon
_{k}}^{\prime }(u_{k})[\tilde{\varphi}_{k}]|.
\end{equation*}
 By Ekeland principle (Remark \ref{ekeland})
and by (\ref{normaphik}) we have that
\begin{equation*}
|F_{\varepsilon _{k}}^{\prime }(w_{k})[\varphi ]|=
|J_{\varepsilon _{k}}^{\prime }(u_{k})[\tilde{\varphi}_{k}]|\leq
\sqrt{\delta _{k}}
\left\vert \left\vert \left\vert \tilde{\varphi}_{k}\right\vert \right\vert
\right\vert _{\varepsilon _{k}}\rightarrow 0\text{ as }k\rightarrow \infty .
\end{equation*}

At this point to get the claim it is sufficient to show that
\begin{equation}
F_{\varepsilon _{k}}^{\prime }(w_{k})[\varphi ]\rightarrow \left(
E^{+}\right) ^{\prime }(w)[\varphi ].  \label{Fprimo}
\end{equation}
In fact we have
\begin{equation*}
\left\vert F_{\varepsilon _{k}}^{\prime }(w_{k})[\varphi ]-\left(
E^{+}\right) ^{\prime }(w)[\varphi ]\right\vert \leq I_{1}+I_{2}+I_{3},
\end{equation*}
where
\begin{eqnarray*}
I_{1} &=&\int_{\text{supp}\varphi }\left(
\sum_{ij}g_{q_{k}}^{ij}(\varepsilon _{k}z)\frac{\partial w_{k}}{\partial
z_{i}}(z)\frac{\partial \varphi }{\partial z_{j}}(z)\left\vert
g_{q_{k}}(\varepsilon _{k}z)\right\vert ^{1/2}-\delta _{ij}\frac{\partial w}
{\partial z_{i}}(z)\frac{\partial \varphi }{\partial z_{j}}(z)\right) dz; \\
I_{2} &=&\int_{\text{supp}\varphi }|\varphi (z)|\left\vert
g_{q_{k}}(\varepsilon _{k}z)\right\vert ^{1/2}|w_{k}(z)-w(z)|dz; \\
I_{3} &=&\int_{\text{supp}\varphi }|\varphi (z)|\left\vert
g_{q_{k}}(\varepsilon _{k}z)\right\vert ^{1/2}|\left( w_{k}^{+}(z)\right)
^{p-1}-(w(z))^{p-1}|dz.
\end{eqnarray*}
Because supp$\varphi $ is a compact set,
$|g_{q_{k}}^{ij}(\varepsilon _{k}z)-\delta _{ij}| \leq
c\varepsilon_{k}\left\vert z\right\vert ^{2}$
and by Step 1 we get (\ref{Fprimo}).

\textsc{Step3:} {\em The limit function $w$ is a least energy solution of}
\begin{equation*}
\left\{
\begin{array}{cc}
-\Delta w+w=(w^{+})^{p-1} & \text{in }\mathbb{R}_{+}^{n}; \\
\frac{\partial w}{\partial \nu }=0 & \text{for }y=(\bar{y},0);
\end{array}
\right.
\end{equation*}

We will show that $w\neq 0$. We are in the case $P_{k}\cap \partial M\neq
\emptyset $. We can choose $T>0$ such that
\begin{equation*}
P_{k}\subset I_{q_{k}}(\varepsilon _{k}T,\varepsilon _{k}T)\text{
for }k\text{ large enough.}
\end{equation*}
where $q_k$ is a point in $P_{k}$.
By definition $w_{k}$ and by Lemma \ref{lemmagamma} there exist
a $q_k$ such that, for $k$ large
enough
\begin{eqnarray*}
\left\vert \left\vert w_{k}^{+}\right\vert \right\vert
_{L^{p}(B_{n-1}(0,T)\times [ 0,T])} &=&\int_{B_{n-1}(0,T)\times
[ 0,T]}\left\vert \chi _{R}(\varepsilon _{k}|\bar{z}|)\chi _{\rho
}(\varepsilon _{k}z_{n})u_{k}^{+}\left( \psi _{q_{k}}^{\partial
}(\varepsilon _{k}z)\right) \right\vert ^{p}dz= \\
&=&\frac{1}{\varepsilon _{k}^{n}}\int_{B_{n-1}(0,\varepsilon _{k}T)\times
[0,\varepsilon _{k}T]}\left\vert u_{k}^{+}\left( \psi
_{q_{k}}^{\partial }(y)\right) \right\vert ^{p}dy\geq  \\
&\geq &\frac{c}{\varepsilon _{k}^{n}}\int_{B_{n-1}(0,\varepsilon
_{k}T)\times [ 0,\varepsilon _{k}T]}\left\vert u_{k}^{+}\left( \psi
_{q_{k}}^{\partial }(y)\right) \right\vert ^{p}\left\vert
g_{q_{k}}(y)\right\vert ^{1/2}dy= \\
&\geq &\frac{c}{\varepsilon _{k}^{n}}\int_{I_{q_{k}}(\varepsilon
_{k}T,\varepsilon _{k}T)}\left\vert u_{k}^{+}\right\vert ^{p}d\mu _{g}\geq
c\gamma >0.
\end{eqnarray*}
Since $w_{k}$ converge strongly to $w$ in $L^{p}(B_{n-1}(0,T)\times
[0,T])$, we have $w\neq 0$.

We now show that
\begin{equation*}
\left( \frac{1}{2}-\frac{1}{p}\right) |w^{+}|_{p}^{p}\leq m_{e}^{+}.
\end{equation*}
Since $u_{k}\in \mathcal{N}_{\varepsilon _{k}}\cap J_{\varepsilon
_{k}}^{m_{e}^++\delta _{k}}$, it holds
\begin{eqnarray*}
\frac{m_{e}^++\delta _{k}}{\frac{1}{2}-\frac{1}{p}} &\geq &\frac{1}{\frac{1}{2}
-\frac{1}{p}}J_{\varepsilon _{k}}(u_{k})=\frac{1}{\varepsilon _{k}^{n}}
\int_{M}|u_{k}^{+}|^{p}d\mu _{g}\geq  \\
&\geq &\frac{1}{\varepsilon _{k}^{n}}\int_{B_{n-1}(q_{k},R/2)\times
[0,\rho /2]}|u_{k}^{+}(\psi _{q_{k}}^{\partial
}(y))|^{p}|g_{q_{k}}(y)|^{1/2}dy= \\
&=&\int_{B_{n-1}(q_{k},R/2\varepsilon _{k})\times [ 0,\rho
/2\varepsilon _{k}]}|u_{k}^{+}(\psi _{q_{k}}^{\partial }(\varepsilon
_{k}z))|^{p}|g_{q_{k}}(\varepsilon _{k}z)|^{1/2}dz.
\end{eqnarray*}
We set
\begin{equation*}
f_k(z)=u_{k}^{+}(\psi _{q_{k}}^{\partial }(\varepsilon
_{k}z))|g_{q_{k}}(\varepsilon _{k}z)|^{1/2}\zeta _{k}(z)
\end{equation*}
where $\zeta _{k}$ is the characteristic function of the set
$B_{n-1}(q_{k},R/\varepsilon _{k})\times [0,\rho /\varepsilon _{k}]$.
The sequence of function $f_{k}$ is bounded in $L^{p}(\mathbb{R}_{+}^{n})$,
hence, up to subsequence, converges weakly to some
$f\in L^{p}(\mathbb{R}_{+}^{n})$. We get, for any $\varphi \in C_{0}^{\infty }(\mathbb{R}_{+}^{n})$,
\begin{equation*}
\int_{\mathbb{R}_{+}^{n}}f_{k}(z)\varphi (z)dz\rightarrow
\int_{\mathbb{R}_{+}^{n}}w^{+}(z)\varphi (z)dz\text{ as }k\rightarrow \infty .
\end{equation*}
Hence $f$ is equal to the positive function $w^+=w\neq 0$. Moreover we have
\begin{equation*}
\left( \frac{1}{2}-\frac{1}{p}\right) |w|_{p}^{p}\leq \liminf_{k\rightarrow
\infty }\left( \frac{1}{2}-\frac{1}{p}\right)
\int_{\mathbb{R}_{+}^{n}}|f_{k}(z)|^{p}dz\leq m_{e}^{+}.
\end{equation*}
Concluding $w\in \mathcal{N}^{+}$ and $E^{+}(w)\leq m_{e}^{+}$, so $w$ is a
least energy solution.

\textsc{Conclusion of the first case:} {\em At this point we can show that,
for any $T>0$, it holds, for $k$ large enough,}
\begin{equation*}
\left( \frac{1}{2}-\frac{1}{p}\right) |w_{k}|_{L^{p}(B_{n-1}(0,T)\times
[ 0,T])}^{p}\leq \frac{2}{3}(1-\eta )m_{e}^{+}.
\end{equation*}
In fact we recall that for any
$q\in \partial M$ the Riemannian metric $g_{q}(y)$
read through the Fermi coordinates is such that $g_{q}(\varepsilon
_{k}z)=1+O(\varepsilon _{k}|z|)$. Hence fixed $T$
\begin{equation*}
|g_{q}(\varepsilon _{k}z)|^{-1/2}\leq \frac{2}{3}\text{ for }k\text{ big
enough and for }z\in B_{n-1}(0,T)\times [ 0,T]
\end{equation*}
By this fact, using the definition of $w_{k}$ and (\ref{uk}) we have, for $k$
large,
\begin{eqnarray}
\nonumber
&&|w_{k}^{+}|_{L^{p}(B_{n-1}(0,T)\times [0,T])}^{p} \leq
\int_{B_{n-1}(0,T)\times [ 0,T]}|u_{k}^{+}(\psi _{q_{k}}^{\partial
}(\varepsilon _{k}z))|^{p}|g_{q_{k}}(\varepsilon _{k}z)|^{1/2}\frac{2}{3}dz=
\\
&&=\frac{2}{3}\frac{1}{\varepsilon _{k}^{n}}\int_{I(q_{k},\varepsilon
_{k}T,\varepsilon _{k}T)}|u_{k}^{+}|^{p}d\mu _{g}\leq \frac{2}{3}(1-\eta )
\frac{m_{e}^{+}}{\left( \frac{1}{2}-\frac{1}{p}\right) }.\label{eq17}
\end{eqnarray}
On the other side by Step 3 we have that
\begin{equation*}
E^{+}(w)=\left( \frac{1}{2}-\frac{1}{p}\right) |w|_{p}^{p}=m_{e}^{+}.
\end{equation*}
Now, by Step 1 there exists $T>0$ such that, for $k$ big enough we have
\begin{equation}\label{eq18}
\frac{2}{3}(1-\eta )\frac{m_{e}^{+}}{\left( \frac{1}{2}-\frac{1}{p}\right) }
<|w_{k}^{+}|_{L^{p}(B_{n-1}(0,T)\times [ 0,T])}^{p}.
\end{equation}
By (\ref{eq17}) and by (\ref{eq18}) we have a contradiction.

\textbf{The case }$P_{k}^{\varepsilon }\cap \partial M=\emptyset $. we
choose a point $q_{k}$ interior to $P_{k}^{\varepsilon }$ and we consider
the normal coordinates at $q_{k}\,$. We set $w_{k}(z)$ as
\begin{equation*}
u_{k}(x)\chi _{R}(\exp_{q_{k}}^{-1}(x))=
u_{k}(\exp_{q_{k}} (y))\chi_{R}(y)=
u_{k}(\exp_{q_{k}} (\varepsilon _{k}z))\chi _{R}(\varepsilon _{k}z)=
w_{k}(z).
\end{equation*}
Then $w_{k}\in H_{0}^{1}(B(0,R/\varepsilon _{k}))\subset
H^{1}(\mathbb{R}^{n})$. Arguing as in the previous step, we can establish
some properties of the function $w_{k}$. We omit the proof of single steps.

\textsc{Step 1:} $w_{k}$ is bounded in $H^{1}$ and converge to some $w\in
H^{1}$ weakly $L_{\text{loc}}^{p}$\ in and strongly in $H^{1}$.

\textsc{Step 2:} $w$ is a weak solution of $-\Delta w+w=(w^{+})^{p-1}$ in
$\mathbb{R}^{n}$

\textsc{Step 3:} $w$ is strictly positive, and it is a least energy solution
of $-\Delta w+w=|w|^{p-1}w$, that is
\begin{equation}\label{eq19}
\left( \frac{1}{2}-\frac{1}{p}\right) |w|_{p}^{p}=E(w)=m_{e}=2m_e^+.
\end{equation}
\textsc{Conclusion of the second case:} By (\ref{eq19}) and (\ref{uk})
we have the
contradiction

This concludes the proof.
\end{proof}

\begin{rem}
We point out that in the proof of Proposition \ref{propconc}, by 
Remark \ref{remlimsup} and by (\ref{eqgeq}) we showed that
\begin{equation*}
\lim_{\varepsilon \rightarrow 0}m_{\varepsilon }=m_{e}^{+}.
\end{equation*}
\end{rem}

\section{The map $\beta $}

For any $u\in \mathcal{N}_{\varepsilon }$ we can define its center of mass
as a point $\beta (u)\in \mathbb{R}^{N}$ by
\begin{equation}
\beta (u)=\frac{\displaystyle\int_{M}x|u^{+}(x)|^{p}d\mu _{g}}{\displaystyle
\int_{M}|u^{+}(x)|^{p}d\mu _{g}}.
\end{equation}

The application is well defined on $\mathcal{N}_{\varepsilon }$, since $u\in
\mathcal{N}_{\varepsilon }$ implies $u^{+}\neq 0$. In the following we will
show that if $u\in \mathcal{N}_{\varepsilon }\cap J^{m_{e}^{+}+\delta }$
then $\beta (u)\in (\partial M)_{3\rho }$,using the concentration property
(Prop. \ref{propconc})
of the function $u\in \mathcal{N}_{\varepsilon }\cap J^{m_{e}^{+}+\delta }$
if $\varepsilon $ and $\delta $ are sufficiently small.

\begin{prop}
\label{propbar1}For any
$u\in \mathcal{N}_{\varepsilon }\cap J^{m_{e}^{+}+\delta }$, with $\varepsilon $
and $\delta $ small enough, it holds
\begin{equation*}
\beta (u)\in (\partial M)_{3\rho }
\end{equation*}
\end{prop}

\begin{proof}
Since $m_{\varepsilon }\rightarrow m_{e}^{+}$ and by
Proposition \ref{propconc} we get that for any
$u\in \mathcal{N}_{\varepsilon }\cap J^{m_{e}^{+}+\delta }$
there exists $q\in \partial M$ such that
\begin{equation}
(1-\eta )m_{e}^{+}\leq\left( \frac{1}{2}-\frac{1}{p}\right)
\frac{1}{\varepsilon ^{n}}|u^{+}|_{L^{p}\left( I_{q}(\rho ,R)\right) }^{p}.
\label{bar1}
\end{equation}
Since $u\in \mathcal{N}_{\varepsilon }\cap J^{m_{e}^{+}+\delta }$ we have
\begin{equation}
\left( \frac{1}{2}-\frac{1}{p}\right) \frac{1}{\varepsilon ^{n}}
|u^{+}|_{p,g}^{p}<m_{e}^{+}+\delta .  \label{bar2}
\end{equation}
Then by (\ref{bar1}) and (\ref{bar2}) we get
\begin{equation*}
\int_{I_{q}(\rho ,R)}\frac{|u^{+}|^{p}}{|u^{+}|_{p,g}^{p}}d\mu _{g}\geq
\frac{1-\eta }{1+\frac{\delta }{m_{e}^{+}}}.
\end{equation*}
By definition of $\beta $ we have
\begin{eqnarray*}
|\beta (u)-q| &\leq &\left\vert \int_{I_{q}(\rho ,R)}(x-q)\frac{|u^{+}|^{p}}
{|u^{+}|_{p,g}^{p}}d\mu _{g}\right\vert +\left\vert \int_{M\smallsetminus
I_{q}(\rho ,R)}(x-q)\frac{|u^{+}|^{p}}{|u^{+}|_{p,g}^{p}}d\mu
_{g}\right\vert \leq  \\
&\leq &2\rho +D\left( 1-\frac{1-\eta }{1+\frac{\delta }{m_{e}^{+}}}\right) ,
\end{eqnarray*}
where $D$ is the diameter of the manifold $M$ as a subset of $\mathbb{R}^{n}$.
Choosing $\eta $ and $\delta $ small enough we get the claim.
\end{proof}

\begin{prop}
The composition
\begin{equation*}
\beta \circ \Phi _{\varepsilon }:\partial M\rightarrow (\partial M)_{3\rho}
\subset \mathbb{R}^{n}
\end{equation*}
is well defined and homotopic to the identity of $\partial M$.
\end{prop}

\begin{proof}
By Proposition \ref{propbar1} and \ref{propphi} the map
$\beta \circ \Phi _{\varepsilon }:
\partial M\rightarrow (\partial M)_{\rho(\partial M)}$
is well defined.

To prove that $\beta \circ \Phi _{\varepsilon }:\partial M\rightarrow
(\partial M)_{3\rho }$ is homotopic to the identity it is enough to evaluate
the map
\begin{eqnarray*}
\beta (\Phi _{\varepsilon }(q))-q &=&\frac{\int_{B_{n-1}(0,R)\times
[0,\rho ]}y|V_{\varepsilon }(y)\chi _{R}(|\bar{y}|)\chi _{\rho }(y_{n})|^{p}dy}
{\int_{B_{n-1}(0,R)\times [0,\rho ]}|V_{\varepsilon }(y)
\chi _{R}(|\bar{y}|)\chi _{\rho }(y_{n})|^{p}dy}= \\
&=&\frac{\varepsilon \int_{B_{n-1}(0,R/\varepsilon )\times [ 0,\rho
/\varepsilon ]}z|V(z)\chi _{R}(|\varepsilon \bar{z}|)\chi _{\rho
}(\varepsilon z_{n})|^{p}dz}{\int_{B_{n-1}(0,R/\varepsilon )\times
[0,\rho /\varepsilon ]}|V(z)\chi _{R}(|\varepsilon \bar{z}|)\chi _{\rho
}(\varepsilon z_{n})|^{p}dz}.
\end{eqnarray*}
By the exponential decay of $V$ we get
$|\beta (\Phi _{\varepsilon}(q))-q|<c\varepsilon $, where $c$ is a
constant not depending on $q$.
\end{proof}

\section{The set $T_\varepsilon$}
To finish the proof of Theorem \ref{mainteo}, it remains to show
that there exists a critical point $\bar u $ of $J_\varepsilon$ in 
${\cal N}_\varepsilon$ with 
$m_e^++ \delta < J_\varepsilon (\bar u ) < c_\varepsilon$, 
for bounded constants $c_\varepsilon$. As
explained in Section 3, this is achieved by constructing a set
$T_\varepsilon$ which contains only positive functions, is
contractible in ${\cal N}_\varepsilon\cap
J_\varepsilon^{c_\varepsilon}$ and contains
$\Phi_\varepsilon(\partial M)$.
The process of building the set $T_\varepsilon$ is analogous to 
the process of section 6 of \cite{BBM07}; for clearness we prefer to show it.

To define the set $T_\varepsilon$ we use the functions $Z_{\varepsilon,q}(x)$
as defined in (\ref{zeq}). We recall that $Z_{\varepsilon,q}(x)\in H^1_g(M)$
are positive functions. 
Let $W(x)\in H^1(\R^n_+)$ be any positive function and denote as
usual $W_\varepsilon(x)= W\left(\frac x\varepsilon \right)$. 
For $q_0\in\partial M$
a fixed point on the boundary of $M$
we introduce the functions
\begin{equation}
v_\varepsilon(x) :=
\left\{
\begin{array}{ll}
W_\varepsilon(y(\xi))\tilde\chi(y(\xi)) &\text{if }\xi\in I_{q_0}(R,\rho); \\
 0 & \text{otherwise}
\end{array}
\right.
\end{equation}
where $y(\xi)=\left(\psi^\partial_{q_0}\right)^{-1}$ and
$\tilde\chi(y)=\chi_R(\bar y)\chi_\rho(y_n)$ as in the previous part of the paper.

We define the cone
\begin{equation}
C_\varepsilon:=
\{u(x) := \theta v_\varepsilon(x)+ (1-\theta)Z_{\varepsilon,q}(x)\ : \ \theta\in[0, 1], q \in \partial M.\}
\subset  H^1_g (M).
\end{equation}
By the properties of the map $\Phi_\varepsilon$ proved in Proposition \ref{propphi},
we have that $C_\varepsilon$ is compact and contractible in $H^1_g (M)$. We now
project it on the Nehari manifold ${\cal N}_\varepsilon$ by the factor $t_\varepsilon(u)$
to obtain
\begin{equation}
T_\varepsilon :=\left\{ t_\varepsilon(u)u: u \in C_\varepsilon,
t_\varepsilon^{p-2} (u) =\frac{ |||u|||^2_\varepsilon}{\frac1{\varepsilon^n} |u|^p_{p,g}} \right\}
\subset {\cal N}_\varepsilon .
\end{equation}

We get that $\Phi_\varepsilon(\partial M)\subset T_\varepsilon$, that
$T_\varepsilon$ contains only
positive functions and that it is compact and contractible in ${\cal N}_\varepsilon$.
Hence if we define
\begin{equation*}
c_\varepsilon := \max_{u\in C_\varepsilon} J_\varepsilon( t_\varepsilon(u)u)
\end{equation*}
we get that
$T_\varepsilon \subset {\cal N}_\varepsilon\cap J^{c_\varepsilon}_\varepsilon$.
The last step is to prove the following
proposition.

\begin{prop} There exists a constant $c >0$ such that for $\varepsilon$
small enough it holds $c_\varepsilon < c$.
\end{prop}
\begin{proof}
By the definition of the
Nehari manifold, we recall that for $u\in C_\varepsilon$ it holds

\begin{equation}\label{65}
J_\varepsilon (t_\varepsilon(u)u)
=\left( \frac12 - \frac1p\right)
t^2_\varepsilon(u)|||u|||^2_\varepsilon=
\left( \frac12 - \frac1p\right)
\frac{|||u|||_\varepsilon^\frac{2p}{p-2}}
{\left( \frac{1} {\varepsilon^n} |u|^p_{p,g}\right)^\frac2{p-2}}.
\end{equation}
Arguing as (\ref{eql2}), (\ref{eqlp}), (\ref{eqgrad})
for $v_\varepsilon$ and $W$ instead of $Z_{\varepsilon,q}$ and $V$,
we find that there exists a constant $k_1 > 0$ such that

\begin{equation}\label{66}
|||u|||^2_\varepsilon \leq   ||W||^2_{H^1}+||V||^2_{H^1}+k_1
\end{equation}
for $\varepsilon$ small enough. Moreover, 
for $\varepsilon$ small enough, we find
constants $k_2 > 0$ and $k_3 > 0$ such that
\begin{eqnarray*}
\frac1{\varepsilon^n} |v_\varepsilon |^p_{p,g}\geq |W|_p^p -k_2> 0, &&
\frac1{\varepsilon^n} |Z_{\varepsilon,q} |^p_{p,g}\geq |V|_p^p -k_3> 0.
\end{eqnarray*}
Hence, since $v_\varepsilon$ and $Z_{\varepsilon,q}$ are
positive functions and $\theta\in [0, 1]$, there exists $k_4$ such that
\begin{equation}
\frac{1}{\varepsilon^n} |u|^p_{p,g}\geq
\frac{1}{\varepsilon^n} 
\max\{ |\theta v_\varepsilon |^p_{p,g}, |(1-\theta)Z_{\varepsilon,q} |^p_{p,g}\}
\geq k_4\label{67}
\end{equation}
 for $\varepsilon$ small enough. Putting together (\ref{65}), (\ref{66}) and
(\ref{67}) we get the thesis.
\end{proof}

\end{document}